\documentclass[a4paper,draft]{amsproc}
\usepackage{amssymb}
\usepackage[hyphens]{url} 

\theoremstyle{plain}
 \newtheorem{theorem}{Theorem}[section]
 \newtheorem{proposition}{Proposition}[section]
 
 \newtheorem{corollary}{Corollary}[section]
\theoremstyle{definition}

\theoremstyle{remark}

 \newtheorem{definition}{definition}
 \numberwithin{equation}{section}

\renewcommand{\leq}{\leqslant}
\renewcommand{\geq}{\geqslant}

\title[ Generating function/ M. Goubi]{On a closed form of rational generating functions for polynomials}



\author{\bfseries Mouloud  Goubi} 
\address{Mouloud Goubi\\
Department of Mathematics \\
University of UMMTO RP. 15000\\
Tizi-ouzou, Algeria\\
Laboratoire d'Alg\`ebre et Th\'eorie des Nombres, USTHB Alger}
\email{mouloud.goubi@ummto.dz}

\begin{document}

\vspace{18mm} \setcounter{page}{1} \thispagestyle{empty}

\begin{abstract}
Our goal in this work is to found a closed form for rational
generating functions, these generate a various families of
polynomials and generalized polynomials, in order to get the general
recursive formula satisfied by these polynomials.
\end{abstract}

\maketitle

\section{Introduction}
\label{Sec:1} First we remember the notion of ordinary generating
function \cite{1} of a family of polynomials in one variable.
\begin{definition}
\label{def1} $f\left(x,t\right)$ is an ordinary generating function
if and only if there exist a sequence $P_k(x)$ of polynomials in
$\mathbb{Z}[x]$ and $\delta>0$ a positive real number such that
\begin{equation}
\label{eqdef1} f\left(x,t\right)=\sum_{k\geq0}P_k\left(x\right)t^k,\
~|t|<\delta.
\end{equation}
\end{definition}
In this paper we consider a family of generating functions for these
the polynomials associated obey to the same general recursive
formula. Let the family $\left\{A_0(x), A_1(x), A_2(x),\cdots,
A_m(x),B_0(x)\neq0,B_1(x), B_2(x),\cdots, B_n(x)\right\}$ of
polynomials such that $B_0(x), B_1(x),\cdots, B_n(x)$ are coprime
and the rational function $f(x,t)$ defined by $f(x,t)=A(x,t)/B(x,t)$
where $A(x,t)=\displaystyle\sum_{j=0}^{m}A_j\left(x\right)t^j$ and
$B(x)=\displaystyle\sum_{l=0}^{n}B_l\left(x\right)t^l.$\\
Taking $h(x,t)=-\sum_{l=1}^{n}{B_l(x)/B_0(x)}t^l$ then
$h\left(x,0\right)=0$. Since $h\left(x,t\right)$ as a function of
$t$ is continuous on $\mathbb{R}$ then there exist a constant
$\delta>0$ such that $|h\left(x,t\right)|<1$ for $|t|<\delta$.
Furthermore
\[\frac{1}{1-h\left(x,t\right)}=\sum_{k\geq0}h^{k}\left(x,t\right),\
|t|<\delta\] this result is deduced from the well known identity
\[\frac{1}{1-t}=\sum_{k\geq0}t^k,\ |t|<1.\]
Without lost generality $h^{k}\left(x,t\right)$ can be written in
the following form
\[h^{k}\left(x,t\right)=\displaystyle\sum_{j_1+j_2+\cdots+j_n=k}{k\choose j_1\cdots j_n}\frac{B^{j_1}_1(x)\cdots B^{j_n}_n(x)}{B^k_0(x)}t^{j_1+2j_2\cdots+nj_{n}}\]
where ${k\choose j_1\cdots j_n}$ is the multinomial of order $n$
\[{k\choose j_1\cdots j_n}=\frac{k!}{j_1!j_2!\cdots j_n!}.\]
Finally
\[f(x,t)=A(x,t)\sum_{k\geq0}\displaystyle\sum_{j_1+j_2+\cdots+j_n=k}
{k\choose j_1\cdots j_n}\frac{B^{j_1}_1(x)\cdots
B^{j_n}_n(x)}{B^{k+1}_0(x)}t^{j_1+2j_2\cdots+nj_{n}}\] to be a
generating function; $B^{k+1}_0(x)$ must divides $A_l(x)$ for every
$k\geq0$ and $0\leq l\leq m$ then $B_0(x)$ must be $1$. We conclude
that all rational generating functions are of the form
\[f(x,t)=\frac{\sum_{j=0}^{m}A_j(x)t^j}{\sum_{l=0}^{n}B_l(x)t^l}\]
with $B_0(x)=1.$
\section{Statement of main results}
\label{Sec:2} Let $f\left(x,t\right)$ the rational function of two
variables $x$ and $t$ considered bellow. Denoting $\chi_m$ the
characteristic function of the set $\left\{0,1,\cdots,m\right\}$. It
means that
\begin{eqnarray*}
\chi_m(k) = \left\{
\begin{array}{lll}
1\ ,&\quad \textrm{ if }\ 0\leq k\leq m, \\
0\ , &\quad  \textrm{otherwise}.
\end{array}
\right.
\end{eqnarray*}
The recursive formula of polynomials generated by
$f\left(x,t\right)$ is given in the following theorem.
\begin{theorem}
\label{th1} If $B_0(x)=1$ then $f(x,t)$ generates the family
$\left\{P_k(x),\ k\geq0\right\}$ of polynomials such that
$P_0(x)=A_0(x)$ and
\begin{equation}
\label{equath1}
P_{k}(x)=\chi_m(k)A_k(x)-\sum_{j=1}^{\min\left\{n,k\right\}}B_j(x)P_{k-j}(x),\
k\geq1
\end{equation}
\end{theorem}
\begin{corollary}\label{coro1}
\label{coro1}If $B_0(x)=1$ then
\[g(x,t)=\frac{1}{B\left(x,t\right)}\]
generates the family $\left\{Q_k(x),\ k\geq0\right\}$ of polynomials
such that $Q_0(x)=1$ and
\begin{equation}
\label{equcoro1}
Q_k(x)=-\sum_{j=1}^{\min\left\{n,k\right\}}B_j(x)Q_{k-j}(x),\ ~
k\geq1.
\end{equation}
\end{corollary}
\begin{proof}
Just taking $A_0\left(x\right)=1$ and $A_k(x)=0$ for $k\geq1$, then
$m=0$. After substitution in the formula \eqref{equath1} Theorem
\ref{th1} we get a new family of polynomial $Q_k(x)$ satisfying the
identity \eqref{equcoro1} Corollary \ref{coro1}.
\end{proof}

The family $\left\{P_k(x),\ k\geq0\right\}$ depend only on the
family $\left\{Q_k(x),\ k\geq0\right\}$, it results from the
convolution product of the two families $\left\{A_j(x),\ 0\leq j\leq
m\right\}$ and $\left\{Q_k(x),\ k\geq0\right\}$. Explicit formula is
given in the following proposition.
\begin{proposition}\label{prop1}
\begin{equation}
\label{equaprop1}P_k(x)=\sum_{j=0}^{\min\left\{m,k\right\}}A_j(x)Q_{k-j}(x).
\end{equation}
\end{proposition}
\begin{proof}
\[f\left(x,t\right)=\sum_{j=0}^{m}A_j(x)g\left(x,t\right)t^j\]then
\[f\left(x,t\right)=\sum_{k\geq0}\sum_{j=0}^{m}A_j(x)Q_k\left(x\right)t^{k+j}\]
and
\[f\left(x,t\right)=\sum_{j=0}^{m}\sum_{k\geq j}A_j(x)Q_{k-j}\left(x,t\right)t^{k}\]
which means that
\[f\left(x,t\right)=\sum_{k\geq0}\sum_{j=0}^{\min\left\{m,k\right\}}A_j(x)Q_{k-j}\left(x\right)t^{k}.\]
Furthermore
\[\sum_{k\geq0}P_k(x)t^k=\sum_{k\geq0}\sum_{j=0}^{\min\left\{m,k\right\}}A_j(x)Q_{k-j}\left(x\right)t^{k}.\]
Finally after comparison between the coefficients of $t^k$ in both
sides of the equality we get
\[P_k(x)=\sum_{j=0}^{\min\left\{m,k\right\}}A_j(x)Q_{k-j}\left(x,t\right).\]
\end{proof}

\begin{corollary}
\label{coro2}
\begin{equation}
\label{equacoro2}\chi_m(k)A_k(x)-P_k(x)=\sum_{j=1}^{\min\left\{n,k\right\}}\sum_{l=0}^{\min\left\{m,k-j\right\}}B_j(x)A_l(x)Q_{k-j-l}(x)
\end{equation}
\end{corollary}
\begin{proof}
From \eqref{equath1} Theorem \ref{th1};
\[\chi_m(k)A_k(x)=\sum_{j=0}^{\min\left\{n,k\right\}}B_j(x)P_{k-j}(x)\]
But in means of the identity \eqref{equaprop1} Proposition
\ref{prop1};
\[P_{k-j}(x)=\sum_{l=0}^{\min\left\{m,k-j\right\}}A_l(x)Q_{k-j-l}(x)\]
then
\[\sum_{j=0}^{\min\left\{n,k\right\}}\sum_{l=0}^{\min\left\{m,k-j\right\}}B_j(x)A_l(x)Q_{k-j-l}(x)=\chi_m(k)A_k(x)\]
and the result \eqref{equacoro2} Corollary \ref{coro2} follows.
\end{proof}
Now if $0\leq k\leq m$ we get
\[\sum_{j=0}^{k}\sum_{i=0}^{k-j}B_j(x)A_i(x)Q_{k-j-i}(x)=A_k(x)\]
and
\[\sum_{j=1}^{k}\sum_{i=0}^{k-j}B_j(x)A_i(x)Q_{k-j-i}(x)+\sum_{i=0}^{k}A_i(x)Q_{k-i}(x)=A_k(x)\]
then
\[A_k(x)-P_k(x)=\sum_{j=1}^{k}\sum_{i=0}^{k-j}B_j(x)A_i(x)Q_{k-j-i}(x)\]

\subsection{Proof of Theorem~\ref{th1}}
First we must remember the Cauchy product of a polynomial
$\sum_{k=0}^{n}R_k(x)t^k$ with an entire series
$\sum_{k\geq0}S_k(x)t^k$ to be
\[\left(\sum_{k=0}^{n}R_k(x)t^k\right)\left(\sum_{k\geq0}S_k(x)t^k\right)=\sum_{k\geq0}\left(\sum_{j=0}^{\min\left\{k,n\right\}}R_j(x)S_{k-j}(x)\right)t^k\]
which is an entire series too, for more details about the procedure
we refer to \cite{2}. Now writing
\[f\left(x,t\right)=\sum_{k\geq0}P_k(x)t^k\] then
\[\left(\sum_{k=0}^{n}B_k(x)t^k\right)\left(\sum_{k\geq0}P_k(x)t^k\right)=\sum_{k=0}^{m}A_k(x)t^k\]
and
\[\sum_{k\geq0}\left(\sum_{j=0}^{\min\left\{n,k\right\}}B_j(x)P_{k-j}(x)\right)t^k=\sum_{k=0}^{m}A_k(x)t^k,\]
furthermore
\[\sum_{k\geq0}\left(\sum_{j=0}^{\min\left\{n,k\right\}}B_j(x)P_{k-j}(x)\right)t^k=\sum_{k\geq0}\chi_m(k)A_k(x)t^k.\]
After identification we obtain
\[\chi_m(k)A_k(x)=\sum_{j=0}^{\min\left\{n,k\right\}}B_j(x)P_{k-j}(x),\ k\geq0\]
and the recursive formula \eqref{equath1} in
Theorem~\ref{th1} follows.\\

This identity states the recursive formula of a large families of
polynomials. Including the polynomials generated by functions of the
form
\[\theta\left(x,t\right)=\frac{\sum_{j=0}^{m}A_j(x)t^j}{\left(\sum_{j=0}^{n}B_j(x)t^j\right)^h}\]
where $B_0(x)=1$ and $h>1$ a positive integer. The reason is that
$\left(\sum_{j=0}^{n}B_j(x)t^j\right)^h$ is only polynomial in
$\mathbb{Z}[x,t]$, it takes the following form
\[\left(\sum_{j=0}^{n}B_j(x)t^j\right)^h=\sum_{j=0}^{hn}D_j(x)t^j.\]
with $D_0(x)=1.$\\

The following table gives a few families of polynomials obeying the
general recursive formula \eqref{equath1} Theorem \ref{th1}.
\begin{table}[htb]
 \caption{few families $P_k(x)$ of polynomials}\label{table1}
  \vglue2mm
\centering
 {
 \begin{tabular}{|l|l|l|l|}
  \hline
  $A(x,t)$ & $B(x,t)$ & Polynomial & \multicolumn{1}{|l|}{Recursive formula}\\\hline
  $t$ & $1-xt-t^2$ & Fibonacci & $F_k(x)-xF_{k-1}(x)-F_{k-2}(x)=0, F_0(x)=0, F_1(x)=1$ \\
  $1$ & $1-t+xt^2$ & Catalan \cite{3} & $C_k(x)-C_{k-1}(x)+xC_{k-2}(x)=0, C_0(x)=C_1(x)=1$ \\
  $t$ & $1-xt-t^m$ & G. Fibonacci \cite{1}& $U_{n,m}(x)-xU_{n-1,m}(x)-U_{n-m,m}(x)=0$, $n\geq m$ \\
  $t$ & $1-t-xt^2$ & Jacobsthal & $J_k(x)-J_{k-1}-xJ_{k-2}(x)=0, J_1(x)=J_2(x)=1$ \\
  $1$ & $1-pxt-qt^2$ &Horadam \cite{1}&$A_k(x)-pxA_{k-1}-qA_{k-2}(x)=0, A_0(x)=0, A_1(x)=1$\\
  $1+qt^2$ & $1-pxt-qt^2$ &Horadam \cite{1}& $B_k(x)-pxB_{k-1}-qB_{k-2}(x)=0, B_0(x)=2, B_1(x)=x$\\
  $1$ & $1-2xt-t^2$ &Pell \cite{1}& $P_k(x)-2xP_{k-1}-P_{k-2}(x)=0, P_0(x)=0, P_1(x)=1$\\
  $2x+2t$ & $1-2xt-t^2$ &Pell-Lucas \cite{1}& $Q_k(x)-2xQ_{k-1}-Q_{k-2}(x)=0, Q_0(x)=2,  Q_1(x)=2x$\\
  $2-xt$ & $1-xt-t^m$ &G. Lucas \cite{1}& $V_{n,m}(x)-xV_{n-1,m}(x)-V_{n-m,m}(x)=0$, $n\geq m$\\ \hline
 \end{tabular}}
\end{table}

\section{Application to generalized Catalan and Fibonacci polynomials}

In the literature a large families of polynomials obey to the
general recursive formula given in Theorem \ref{th1}. In this
section, we revisit the works \cite{1,2} and get, with a new method
the recursive formulas satisfied by generalized Catalan polynomials
and Fibonacci polynomials in the case $\lambda=1$ for the first and
$h=1$ for the second.\\

Firstly we began by generalized Catalan polynomials, these generated
by the function
\[\frac{1+A(x)t}{1-mt+xt^m}=\sum_{k\geq0}\mathcal{P}^{1,A}_{k,m}(x)t^k\]
According to Theorem \ref{th1}; the following proposition states
exactly the same recursive formula of the family
$\mathcal{P}^{1,A}_{k,m}(x)$ as in Corollary 2.1 \cite{3} p. 166 by
taking $\lambda=1$.
\begin{proposition}
\label{prop2} The family $\left\{\mathcal{P}^{1,A}_{k,m}(x),\
k\geq0\right\}$ is defined by \[\mathcal{P}^{1,A}_{0,m}(x)=1,
\mathcal{P}^{1,A}_{1,m}(x)=A(x)+m,\]
\begin{equation}
\label{equa1prop2}\mathcal{P}^{1,A}_{k,m}(x)=m\mathcal{P}^{1,A}_{k-1,m}(x),\
2\leq k< m,
\end{equation}
and
\begin{equation}
\label{equa2prop2}\mathcal{P}^{1,A}_{k,m}(x)=m\mathcal{P}^{1,A}_{k-1,m}(x)-x\mathcal{P}^{1,A}_{k-m,m}(x),\
k\geq m.
\end{equation}
\end{proposition}
\begin{proof}. In means of the identity \eqref{equath1} Theorem \ref{th1} we deduce that
\[\mathcal{P}^{1,A}_{0,m}(x)=1, \mathcal{P}^{1,A}_{1,m}(x)=A(x)+m\]
and
\[\mathcal{P}^{1,A}_{k,m}(x)=-\sum_{j=1}^{\min\left\{m,k\right\}}B_j(x)\mathcal{P}^{1,A}_{k-j,m}(x),\ k\geq2.\]
with $B_1(x)=-m, B_m(x)=x$ and the others are zero. Explicitly
\[\mathcal{P}^{1,A}_{k,m}(x)=m\mathcal{P}^{1,A}_{k-1,m}(x),\  2\leq k< m\]
and
\[\mathcal{P}^{1,A}_{k,m}(x)=m\mathcal{P}^{1,A}_{k-1,m}(x)-x\mathcal{P}^{1,A}_{k-m,m}(x),\  k\geq m\]
\end{proof}

Without lost generality the identity \eqref{equath1} Theorem
\ref{th1} can be adapted to polynomials of two variables in the
following way
\[f(x,y,t)=\frac{A(x,y,t)}{B(x,y,t)}\] with
$A(x,y,t)=\sum_{j=0}^{n}A_j(x,y)t^j$ and
$B(x,y,t)=\sum_{j=0}^{m}B_j(x,y)t^j.$ With the same demarche we
conclude that $f(x,y,t)$ is a generating function if and only if
$B_0(x,y)=1.$ In this case \[f(x,y,t)=\sum_{k\geq0}P_k(x,y)t^k\] and
the corresponding recursive formula is
\begin{equation}
\label{bivariate}P_{k}(x,y)=\chi_m(k)A_k(x,y)-\sum_{j=1}^{\min\left\{n,k\right\}}B_j(x,y)P_{k-j}(x,y)
\end{equation}

Secondly, the function
\[f(x,y,t)=\frac{1+A(x,y)t}{1-x^kt-y^mt^{m+n}}=\sum_{k\geq0}\mathcal{G}^{1,A}_{\nu}\left(x,y,k,m,n\right)t^{\nu}\] generates the
generalized two variables Fibonacci polynomials
$\mathcal{G}^{1,A}_{\nu}\left(x,y,k,m,n\right)$ and the same result
as in Proposition 4.2 \cite{2} is deduced where
\[\mathcal{G}^{0,A}_{1}\left(x,y,k,m,n\right)=1, \mathcal{G}^{1,A}_{1}\left(x,y,k,m,n\right)=A(x,y)+x^k,\]
\begin{equation*}
\mathcal{G}^{1,A}_{\nu}\left(x,y,k,m,n\right)=x^k\mathcal{G}^{1,A}_{\nu-1}\left(x,y,k,m,n\right),\
2\leq \nu<n+m
\end{equation*}
and
\begin{equation*}
\mathcal{G}^{1,A}_{\nu}\left(x,y,k,m,n\right)=x^k\mathcal{G}^{1,A}_{\nu-1}\left(x,y,k,m,n\right)+y^m\mathcal{G}^{1,A}_{\nu-n-m}\left(x,y,k,m,n\right),\
\nu\geq n+m.
\end{equation*}
These two kinds of polynomials admit a natural generalization to the
forms
\[\frac{\sum_{j=0}^{m}A_j(x)t^j}{\left(1-mt+xt^m\right)^h}=\sum_{k\geq0}\mathcal{P}^{h,A}_{k,m}(x)t^k\]
and
\[\frac{\sum_{j=0}^{m}A_j(x)t^j}{\left(1-x^kt-y^mt^{m+n}\right)^h}=\sum_{k\geq0}\mathcal{G}^{h,A}_{k}\left(x,y,k,m,n\right)t^k.\]
The recursive formula satisfied by these polynomials is left as an
exercise.
\section{Conclusion}
\label{Sec:3} This method is efficient, it gives directly the
recursive formula of infinitely many families of polynomials
generated by rational functions. An open question is: can we found a
general explicit formula for these polynomials?

\end{document}